\documentclass[11pt,bezier]{article}
\usepackage{amsmath, amssymb, amsfonts, graphicx}

\newtheorem{prethm}{{\bf Theorem}}

\newenvironment{thm}{\begin{prethm}\sl{\hspace{-0.5
               em}{\bf.}}}{\end{prethm}}

\newtheorem{prepro}[prethm]{{\bf Proposition}}

\newtheorem{prelem}[prethm]{{\bf Lemma}}

\newenvironment{lem}{\begin{prelem}\sl{\hspace{-0.5
               em}{\bf.}}}{\end{prelem}}

\newtheorem{predeff}[prethm]{{\bf Definition}}

\newtheorem{precor}[prethm]{{\bf Corollary}}

\newenvironment{cor}{\begin{precor}\sl{\hspace{-0.5
               em}{\bf.}}}{\end{precor}}

\newtheorem{preconj}[prethm]{{\bf Conjecture}}

\newenvironment{conj}{\begin{preconj}\sl{\hspace{-0.5
               em}{\bf.}}}{\end{preconj}}

\newtheorem{preremark}[prethm]{{\bf Remark}}

\newenvironment{remark}{\begin{preremark}\rm{\hspace{-0.5
               em}{\bf.}}}{\end{preremark}}

\newtheorem{preexample}[prethm]{{\bf Example}}

\newtheorem{preproof}{{\bf\textsf{Proof.}}}

\newenvironment{proof}[1]{\begin{preproof}{\rm
               #1}\hfill{$\Box$}}{\end{preproof}}

\newcommand{\la}{\lambda}
\newcommand{\x}{{\bf x}}
\newcommand{\y}{{\bf y}}
\newcommand{\mul}{{\rm mult}}

\title{  Spectral properties of cographs and $P_5$-free graphs}

\author{{\sc Ebrahim Ghorbani} \\[.3cm]
{\sl Department of Mathematics, K.N. Toosi University of Technology,}\\
{\sl P. O. Box 16315-1618, Tehran, Iran}\\
{\sl School of Mathematics, Institute for Research in Fundamental
Sciences (IPM),}\\
{\sl P.O. Box
19395-5746, Tehran, Iran }
\\[.3cm]
$\mathsf{e\_ghorbani@ipm.ir}$ }

\begin{document}
\maketitle

\vspace{5mm}

\begin{abstract}
A cograph is a simple graph which contains no path on 4 vertices as an induced subgraph.
We consider the eigenvalues of adjacency matrices of cographs and prove that a graph $G$ is a cograph if and only if no induced subgraph of $G$ has an eigenvalue in the interval $(-1,0)$.
It is also shown that  the multiplicity of any eigenvalue of a cograph $G$ does not exceed the sum of multiplicities of $0$ and $-1$ as eigenvalues of $G$.
We introduce a partial order on the vertex set of graphs $G$ in terms of inclusions among the open and closed neighborhoods of vertices,
and conjecture that the multiplicity of any eigenvalue of a cograph $G$  except for $0,-1$ does not exceed  the maximum size of an antichain with respect to that partial order. In two extreme cases (in particular for threshold graphs), the conjecture is shown to be true.
Finally, we give a simple proof for the result that bipartite $P_5$-free graphs have no eigenvalue in the intervals $(-1/2,0)$ and $(0,1/2)$.

\vspace{5mm}
\noindent {\bf AMS Mathematics Subject Classification\,(2010):}   05C50, 05C75
\end{abstract}

\vspace{5mm}

\section{Introduction}

A cograph is a simple graph which contains no path on four vertices as an induced subgraph.
The family of cographs is the
smallest class of graphs that includes the single-vertex graph and is closed under
complementation and disjoint union.
`Cograph' stands for `complement reducible graph,' a name which was coined in \cite{clb}.
However, this family of graphs were initially defined under different names \cite{j,l,se,su} and since then have been intensively studied.
It is well known that any cograph has a canonical tree representation, called a cotree.
This tree decomposition scheme of cographs is a particular case of the modular decomposition \cite{g} that applies to arbitrary graphs.
Partly because of this property, cographs are interesting from the algorithmic perspective (see \cite{bls}).
As pointed out in \cite{sa}, cographs have numerous applications in areas like parallel computing \cite{noz} or even biology \cite{gkbc} as they can be used to model series-parallel decompositions. For an account on different characterization and properties of cographs see \cite{bls}.

Cographs have also been studied from an algebraic perspective. Based on a computer search, Sillke \cite{si} conjectured that
 the rank of the adjacency matrix of any cograph is equal to the number of distinct
non-zero rows in this matrix. The conjecture was proved by Royle \cite{roy}. Since then alternative proofs and extensions of this result have appeared \cite{bss,chy,htw,sa}.

In this paper we explore further properties of the eigenvalues of the adjacency matrix of a cograph.
We present a new characterization of cographs, namely a graph $G$ is a cograph if and only if no induced subgraph of $G$ has an eigenvalue in the interval $(-1,0)$.  We also show that  the multiplicity of any eigenvalue of a cograph $G$ does not exceed
the total number of duplication and coduplication classes of $G$ (see Section~\ref{pre} for definition) which is not greater than
the sum of multiplicities of $0$ and $-1$ as eigenvalues of $G$. We introduce a partial order on the vertex set of graphs $G$ in terms of inclusions among the open and closed neighborhoods of vertices.
We conjecture that the multiplicity of any eigenvalue of a cograph $G$  except for $0,-1$ does not exceed  the maximum size of an antichain  with respect to that partial order.  We prove the conjecture in two extreme cases: when all vertices are comparable with respect to the partial order (i.e.  the graph is a threshold graph), and when no two vertices are comparable. As a natural extension of $P_4$-graphs, we consider $P_5$-free graphs in Section 6 and show that bipartite $P_5$-free graphs have no eigenvalue in the intervals $(-1/2,0)$ and $(0,1/2)$.

\section{Preliminaries}\label{pre}

In this section we introduce the notation and recall  a basic result  which will be used frequently.
The graphs we consider are all simple and undirected.
For a  graph $G$, we denote  by $V(G)$ the vertex set of $G$. The {\em order} of $G$ is $|V(G)|$.
For two vertices $u,v$, by $u\sim v$ we mean $u$ and $v$ are adjacent.
 If  $V(G)=\{v_1, \ldots , v_n\}$, then the {\em adjacency matrix} of $G$ is an $n \times  n$
 matrix $A(G)$ whose $(i, j)$-entry is $1$ if $v_i\sim v_j$ and  $0$ otherwise.
 By {\em eigenvalues} and {\em rank} of $G$ we mean those of $A(G)$.
 The multiplicity of an eigenvalue $\la$ of $G$ is denoted by $\mul(\la,G)$.
For a vertex $v$ of $G$, let $N_G(v)$ denote the {\em open neighborhood} of $v$, i.e.   the set of
all vertices of $G$ adjacent to $v$ and $N_G[v]=N_G(v)\cup\{v\}$ denote the {\em closed neighborhood} of $v$; we will drop
the subscript $G$  when it is clear from the context.
Two vertices $u$ and $v$ of $G$ are called {\em duplicate} if $N(u)=N(v)$
and called {\em coduplicate} if $N[u]=N[v]$. Note that duplicate vertices cannot be
adjacent and coduplicate vertices must be adjacent. Also, as in \cite{clb}, we say that $u$ and $v$ are {\em siblings} if they are either duplicates or coduplicates.
A maximal subset $S$ of  $V(G)$ with $|S|>1$   such that $N(u)=N(v)$
for any $u, v\in  S$  is called  a  {\em duplication  class} of $G$.  Coduplication  classes and sibling classes are defined analogously.
If $X\subset V(G)$, we use the notation $G-X$ to mean the subgraph of $G$ induced by $V(G)\setminus X$.

An important subclass of cographs are {\em threshold graphs}. These are the graphs which are both a cograph and a split graph (i.e. their vertex sets can be partitioned into a clique and an independent set). For more information see \cite{bls,mp}.

We will make use of the interlacing property of graph eigenvalues which we recall below (see \cite[Theorem~2.5.1]{bh}).
\begin{lem}\label{inter}
Let $G$ be a graph of order $n$, $H$ be an induced subgraph of $G$ of order $m$, $\la_1\ge\cdots\ge\la_n$ and $\mu_1\ge\cdots\ge\mu_m$ be the eigenvalues of $G$ and $H$, respectively. Then $$\la_i\ge\mu_i\ge\la_{n-m+i}~~\hbox{for}~ i=1, \ldots,m.$$
In particular, if $m=n-1$, then
 $$\la_1\ge\mu_1\ge\la_2\ge\mu_2\ge\cdots\ge\la_{n-1}\ge\mu_{n-1}\ge\la_n.$$
\end{lem}
From the case of equality in interlacing (see \cite[Theorem~2.5.1]{bh}) the following can be deduced.
\begin{lem}\label{eqinter} If in Lemma~\ref{inter}, we have $\la_i=\mu_i$ or $\mu_i=\la_{n-m+i}$ for some $1\le i\le m$, then $A(H)$ has an eigenvector $\x$ for $\mu_i$, such that  $\begin{pmatrix}\bf0 \\ \x\end{pmatrix}$, where the entries of the $\bf0$ vector correspond to vertices in $V(G)\setminus V(H)$, is an eigenvector of $A(G)$ for the eigenvalue $\mu_i$.
\end{lem}

\section{A new characterization of cographs}

Several characterizations are known for cographs \cite{bls}. In this section, we present a new characterization of cographs which is based on graph eigenvalues.
We first state the following useful characterization of cographs (see \cite[Theorem~11.3.3]{bls}).

\begin{lem}{\label{dupcodup}} A graph $G$ is a cograph if and only if every induced subgraph of $G$ with more than one vertex has a pair of sibling vertices.
\end{lem}

Now we are in the position to state and prove our new characterization of cographs.
We remark that the {\em only if} part of this was independently
proven in \cite{mt} using a different argument.

\begin{thm}\label{(-1,0)} A graph $G$ is a cograph if and only if no induced subgraph of $G$ has an eigenvalue in the interval $(-1,0)$.
\end{thm}
\begin{proof}{Among the graphs on four vertices, the 4-vertex path is the only one with an eigenvalue in the interval $(-1,0)$ (see \cite[p.~17]{bh}).
This implies that if $G$ has no induced subgraph on four vertices with an eigenvalue in $(-1,0)$, then $G$ is a cograph.

Conversely, assume that $G$ is a cograph.
Since any induced subgraph of a cograph is also a cograph, it suffices to prove the assertion for $G$ itself.
We proceed by induction on $n$, the order of $G$.
The assertion holds if $n\le3$ as no graph with $n\le3$ vertices has an eigenvalue in $(-1,0)$. Let $n\ge4$.
By Lemma~\ref{dupcodup}, $G$ has either a pair of duplicates or a pair of coduplicates.

First assume that $G$  has a pair of duplicates $u,v$ and $H=G-v$.
Let $\la_1\ge\cdots\ge\la_n$ and $\mu_1\ge\cdots\ge\mu_{n-1}$ be the eigenvalues of $G$ and $H$, respectively.
Also suppose that $\mu_t>\mu_{t+1}=\cdots=\mu_{t+j}=0>\mu_{t+j+1}$ (with possibly $j=0$).
Since $H$ is also a cograph, by the induction hypothesis, $\mu_{t+j+1}\le-1$.
By interlacing, we have $\la_{t+1}\ge0=\la_{t+2}=\cdots=\la_{t+j}=0\ge\la_{t+j+1}\ge\mu_{t+j+1}$.
Note that in $A(G)$, the rows corresponding to $u$ and $v$ are the same, so $G$ and $H$ have the same rank which means that
  $\mul(0,G)=\mul(0,H)+1=j+1$. This is  possible only if both $\la_{t+1}$ and $\la_{t+j+1}$ are zero.
On the other hand, again by interlacing, $\la_t\ge\mu_t>0$ and $-1\ge\mu_{t+j+1}\ge\la_{t+j+2}$. Hence $G$ has no eigenvalue in $(-1,0)$.

Next, let $u,v$ be a pair of coduplicates of $G$.
Suppose that $\mu_r>\mu_{r+1}=\cdots=\mu_{r+k}=-1>\mu_{r+k+1}$ (with possibly $k=0$).
Since $H$ is also a cograph, by the induction hypothesis, $\mu_r\ge0$.
By interlacing, we have $\la_{r+1}\ge-1=\la_{r+2}=\cdots=\la_{r+k}=-1\ge\la_{r+k+1}$.
Note that in the matrix $A(G)+I$, the rows corresponding to $u$ and $v$ are the same, so $A(G)+I$ and $A(H)+I$ have the same rank which means that
  $\mul(-1,G)=\mul(-1,H)+1=k+1$. This is  possible only if $\la_{r+1}=\la_{r+k+1}=-1$.
On the other hand, again by interlacing, $\la_r\ge\mu_r\ge0$ and $-1>\mu_{r+k+1}\ge\la_{r+k+2}$. So $G$ has no eigenvalue in $(-1,0)$.
}\end{proof}

As a refinement of Theorem~\ref{(-1,0)}, one may wonder that whether it is true that ``if a graph $G$ has no eigenvalues in the interval $(-1,0)$, then $G$ is a cograph.'' However, this is not true in general, for the graph $P_5$ has no eigenvalues in the interval $(-1,0)$ (see Table 1 in Appendix of \cite{cds}) but it is not a cograph.

We remark that in \cite{jtt} it was shown that threshold graphs (a subclass of cographs) have no eigenvalues in $(-1, 0)$.

\section{Multiplicity of eigenvalues}

This section deals with eigenvalue multiplicity in a cograph in relation with the duplication and coduplication classes.
While it is known that in a cograph $G$ the multiplicities of $0$ and $-1$ eigenvalues can be determined by the sizes of sibling classes (see Lemma~\ref{redcored} below),
here we show that for any eigenvalue $\la\ne0,-1$ of $G$,
  $\mul(\la,G)$  does not exceed the total number of duplication and coduplication classes.
  This in turn implies that  $\mul(\la,G)\le\mul(0,G)+\mul(-1,G)$.

 We begin with the following two lemmas.
\begin{lem}\label{NoIntersect} In any graph, any duplication class has no intersection with any coduplication class.
\end{lem}
\begin{proof}{Let $G$ be a graph and $D,C$ be a duplication and a coduplication class of $G$, respectively, such that
$v\in C\cap D$. So, there are $u\in C$ and $w\in D$ such that $N[u]=N[v]$ and $N(w)=N(v)$. As $u\ne v$, $u\in N(v)=N(w)$, which means
$u\sim w$ and thus $v\sim w$, which is a contradiction as two duplicates cannot be adjacent.
}\end{proof}

\begin{lem}\label{Class} Let $G$ be a cograph, $C$ a sibling class in $G$, $C'\subset C$, $H=G-C'$, and $u,v\in V(H)$.
\begin{itemize}
  \item[\rm(i)] If $C\ne C'\cup\{u\}$, then $u,v$ are duplication (resp., coduplication) pair in $H$ if and only if they are  duplication (resp., coduplication) pair in $G$
 \item[\rm(ii)] If  $C=C'\cup\{u\}$ and $u,v$ are duplication (resp., coduplication) pair in $H$, then $C$ is a coduplication (resp., duplication) class in $G$.
\end{itemize}
\end{lem}
\begin{proof}{ For (i), we assume that $C$ is a duplication class, the proof of the other case is similar.
If both $u,v$ belong to $C\setminus C'$ or both do not belong to $C\setminus C'$, then the assertion is obvious.
So we may assume that $u\in C\setminus C'$ and $v\not\in C$ which means that $N_G(u)\ne N_G(v)$ and $N_G[u]\ne N_G[v]$ (by Lemma~\ref{NoIntersect}).
We show that $u,v$ are neither duplicates nor coduplicates in $H$. 
If $N_H(u)=N_H(v)$, it turns out that there must be a $w\in C'$ such that  $w\in N_G(v)$ (and of course $w\not\in N_G(u)$).
This means that $v\in N_G(w)$ but $v\not\in N_G(u)$ which is a contradiction as $u$ and $w$ are duplicates in $G$.
As $C\ne C'\cup\{u\}$,  there is a vertex $z$ other than $u$ in $C\setminus C'$.
So $u$ and $z$ are duplicates in $H$. Hence, by Lemma~\ref{NoIntersect}, $u$ does not have any coduplicate in $H$, and in particular, $N_G[u]\ne N_G[v]$.

To prove (ii), if $u,v$ are duplicates in $H$,  similar to the above argument, the assumption that $C$ is a duplication class yields a contradiction.
So $C$ must be a coduplication class. The proof of the other case is similar.
}\end{proof}

\begin{remark}\label{rem} Let $\x$ be an eigenvector for eigenvalue $\la$ of a graph $G$. Then the entries of $\x$ satisfy the following equalities:
\begin{equation}\label{sumrule}
\la\x(v)=\sum_{u\sim v}\x(u),~~\hbox{for all}~v\in V(G).
\end{equation}
From this it is seen that if $\la\ne0$, then $\x$ is constant on each duplication class and if $\la\ne-1$, then $\x$ is constant on each coduplication class.
\end{remark}

Now, we are in a position to prove the main result of this section.

\begin{thm}\label{mult} In a cograph, the multiplicity of any eigenvalue except for $0,-1$   does not exceed the total number of sibling classes.
\end{thm}
\begin{proof}{Let $G$ be a cograph and $\la\ne0,-1$ be an eigenvalue of $G$.
Let  $S_1,\ldots,S_k$ be all the sibling classes of $G$.
Assume for a contradiction that $\mul(\la,G)>k$.

We claim that there is an eigenvector for $\la$ which is zero on $S_1\cup\cdots\cup S_k$.
 To see this, from each class, we pick one vertex and remove them from  $G$ to obtain $G'$.
At least $k+1$ consecutive eigenvalues of $G$, say $\la_t,\ldots,\la_{t+k}$, are all equal to $\la$.
Let $\mu_1\ge\cdots\ge\mu_{n-k}$ be the eigenvalues of $G'$. By interlacing, $\la=\la_t\ge\mu_t\ge\la_{t+k}=\la$. So we have equality in interlacing.
Then by Lemma~\ref{eqinter}, if $\y\ne\bf0$ is an eigenvector of $\la$ for $G'$, then $\x:=\begin{pmatrix}\bf0 \\ \y\end{pmatrix}$ is an eigenvector of $\la$ for $G$, where the entries of the $\bf0$ vector corresponds to the vertices removed from $G$ to obtain $G'$. Since any eigenvector for $\la$ is constant on each sibling class (by Remark~\ref{rem}), $\x$ must be zero on  $S_1\cup\cdots\cup S_k$.

Again consider $G$ where this time from each class, we keep one vertex $v_i\in S_i$, $i=1,\ldots,k$ and remove the rest of vertices of $S_i$'s.
 Let $H_0$ be the resulting graph.
 Note that by considering Equation \eqref{sumrule} for zero and nonzero entries of $\x$, it is seen that $\x$ must have at least two nonzero entries.
 Hence $H_0$ has at least two vertices.
  We claim that any sibling  class of $H_0$ contains some vertex from $\{v_1,\ldots,v_k\}$.
To see this,  consider $G_1=G-(S_1\setminus\{v_1\})$.
In view of Lemmas~\ref{NoIntersect} and \ref{Class}, either $v_1$ does not appear in any sibling class of $G_1$ or if it appears in some sibling class $R$, then
$R=\{v_1,u\}$ for some vertex $u\not\in S_2\cup\cdots\cup S_k$, or $R=S_j\cup\{v_1\}$ for some $2\le j\le k$. (Note that in the latter case  by  Lemma~\ref{Class}, $R$ and $S_j$ are both coduplicate (duplicate) classes if $S_1$ is a duplicate (coduplicate) class.)
By continuing this process in $k$ steps, we remove
\begin{equation}\label{C}(S_1\setminus\{v_1\})\cup\cdots\cup(S_{k}\setminus\{v_{k}\})\end{equation}
from $G$ to obtain $H_0$.
 The above argument shows that at the step in which $S_i\setminus\{v_i\}$ is removed, if a new sibling class was created, then it would contain $v_i$. The claim now follows.

Suppose that $R_1,\ldots,R_\ell$ are all the  sibling classes of $H_0$. After an appropriate relabeling, we may assume that $v_i\in R_i$ for $i=1,\ldots,\ell$.
Note if we remove zero entries from an eigenvector of a graph and remove the corresponding vertices from the graph, the resulting vector is an eigenvector of the same eigenvalue for the resulting graph.
Hence, if $\x_0$ is obtained from $\x$ by removing the entries corresponding with the vertices of \eqref{C}, then $\x_0$ is an eigenvector of $\la$ for $H_0$.
Since $\x_0$ is zero on each $v_i$, and  it is constant on each $R_i$ (by Remark~\ref{rem}), it must be zero on $R_1\cup\cdots\cup R_\ell$.
Again we remove
\begin{equation}\label{C'}
(R_1\setminus\{v_1\})\cup\cdots\cup(R_\ell\setminus\{v_\ell\})
\end{equation}
from $H_0$ and call the resulting graph $H_1$.
 Again, with the same arguments as above, if $\x_1$ is obtained from $\x_0$ by removing the zero entries corresponding with the vertices of \eqref{C'}, then $\x_1$ is an eigenvector of $\la$ for $H_1$.
So far we have that $\x$ is zero on $V(G)\setminus V(H_1)$.  Continuing this process we will end up with some subgraph $H_r$ which consists of some isolated vertices and
we have that $\x$ is zero on $V(G)\setminus V(H_r)$, which implies that $\x=\bf0$, a contradiction.
}\end{proof}

The following result from \cite{sa} (see also \cite{bss,mt}) shows that in a cograph, the sizes of sibling classes determine the multiplicities of $0$ and $-1$ eigenvalues. This also can be viewed as a generalization of Sillke's conjecture.

\begin{lem}\label{redcored} Let $G$ be a cograph,  $C_1,\ldots,C_r$ and $D_1,\ldots,D_s$, with $r,s\ge0$, be all the duplication and coduplication classes of $G$, respectively.
Then $$\mul(0,G)=\sum_{i=1}^r(|C_i|-1),~~\hbox{and}~~\mul(-1,G)=\sum_{i=1}^s(|D_i|-1).$$
In particular, if $G$ has no pairs of duplicates, then $G$ has no eigenvalue $0$ and if $G$ is has no pairs of coduplicates, then $G$ has no  eigenvalue $-1$.
\end{lem}
Since each sibling class has at least two vertices, Lemma~\ref{redcored} implies that in a cograph $G$, $\mul(0,G)+\mul(-1,G)$ is not smaller than the total number of sibling classes.
 So we have the following.

\begin{cor} In any cograph $G$, the multiplicity of any eigenvalue does not exceed $\mul(0,G)+\mul(-1,G)$.
\end{cor}

\section{Chain of neighborhoods}

In this section we first introduce an equivalence relation on the vertices of a graph, and then on the set of equivalence classes, we introduce a partial order
in terms of open/closed neighborhoods of vertices.
We conjecture that the multiplicity of  eigenvalues of a cograph except for $0,-1$ is bounded above by the maximum size of an antichain with respect to this partial order.  We show that the conjecture is true in two extremal cases.

Let $G$ be a graph and consider the following relation on $V(G)$:
 $$u\equiv v~~\hbox{if and only if}~\left\{\begin{array}{ll}N[u]=N[v]&\hbox{if $u\sim v$,}\\N(u)=N(v)&\hbox{if $u\not\sim v$.}\end{array}\right.$$
Then `$\equiv$' is an equivalence relation on $V(G)$. It is trivially reflexive and symmetric.
To check the transitivity, let $u\equiv v$ and $v\equiv w$. If either $u\sim v,v\sim w$ or $u\not\sim v,v\not\sim w$, then obviously $u\equiv w$.
Hence, without loss of generality, we are left with the case  $u\sim v,v\not\sim w$. Then we have $N[u]=N[v]$ and $N(v)=N(w)$.
So $u\in N(v)=N(w)$ which in turn implies that  $w\in N[u]=N[v]$, which is a contradiction.

The equivalence relation `$\equiv$' partitions $V(G)$ into equivalence classes. In fact each equivalence class
is either a set of a single vertex, or a sibling class.
We pick one representative from each equivalence class and denote the resulting set by $G/\!\equiv$.
On $G/\!\equiv$ we define the following relation:
 $$u\leqslant v~~\hbox{if and only if}~\left\{\begin{array}{ll}N[u]\subset N[v]&\hbox{if $u\sim v$,}\\N(u)\subset N(v)&\hbox{if $u\not\sim v$,}\end{array}\right.$$
or equivalently
$$u\leqslant v~~\hbox{if and only if}~~N(u)\subset N[v].$$
We observe that `$\leqslant$' is a partial order on $G/\!\equiv$. We only need to check the transitivity.
Assume that $u\leqslant v$ and $v\leqslant w$. Again if  either $u\sim v,v\sim w$ or $u\not\sim v,v\not\sim w$, then obviously $u\leqslant  w$.
Two cases are left to check. First, $u\not\sim v$ and $v\sim w$. So, we have $N(u)\subset N(v)$ and $N[v]\subset N[w]$, implying that
$N(u)\subset N[w]$, i.e. $u\leqslant w$. Second,   $u\sim v$ and $v\not\sim w$. So, we have $N[u]\subset N[v]$ and $N(v)\subset N(w)$.
Hence, $u\in N(v)\subset N(w)$, and thus $w\in N(u)\subset N[v]$ which means $v\sim w$, a contradiction.

Now consider the chains and antichains of $G/\!\equiv$ with respect to the partial order `$\leqslant $'.
Note that by Dilworth's theorem the minimum
number of disjoint chains needed to partition $G/\!\equiv$ is equal to  the maximum size of an
antichain of $G/\!\equiv$.
If $G/\!\equiv$ contains only antichains of size 1, then `$\leqslant $' is a total order, and all the elements of $G/\!\equiv$ belong to a chain, say
$v_1\leqslant v_2\leqslant \cdots\leqslant v_r$. As the above argument shows $v_i\sim v_{i+1}\not\sim v_{i+2}$ is impossible for any $i$.
 So  there must exist some $1\le j\le r$ such that $v_1\not\sim\cdots\not\sim v_j\sim\cdots\sim v_r$.
It turns out that $G$ is a split graph as the vertices equivalent with any of $v_1,\ldots,v_j$ form an independent set and
  the vertices equivalent with any of $v_{j+1},\ldots,v_r$ form a clique.
  It is known that a split graph such that the neighborhoods of its vertices form a `chain' as above, is a threshold graph
(see \cite[Theorem 1.2.4]{mp}).

In the case that `$\leqslant $' is a total order on $G/\!\equiv$, i.e.  $G/\!\equiv$ itself is  chain,
  a strong constraint is imposed on  eigenvalues  multiplicities: any eigenvalue $\la\ne0,-1$ of a threshold graphs is simple (\cite{jtt0}).
As an extension of this result, we conjecture that in general there would be a relation between the chains in $G/\!\equiv$ and eigenvalues multiplicities.

\begin{conj} For any cograph $G$, the multiplicity of any eigenvalue $\la\ne0,-1$,
does not exceed the minimum number of chains into which $G/\!\equiv$ can be partitioned under
the partial order `$\leqslant $'.
\end{conj}

The aformentioned result of \cite{jtt0} on multiplicity of eigenvalues $\la\ne0,-1$ in threshold graphs shows that the conjecture is true in the extreme case that  $G/\!\equiv$ is a chain.
In the other extreme that $G/\!\equiv$ is an antichain, the conjecture follows from Theorem~\ref{mult}.
The validity of the conjecture has also been confirmed for cographs up to 13 vertices by a computer search.
Cographs up to 13 vertices can be constructed making use of the inductive structure of cographs and then by isomorphism rejection using {\tt nauty} \cite{mckay}.

\section{$P_5$-free graphs}
In this section we consider $P_5$-free graphs as a natural extension of cographs. These are the graphs with no induced subgraph isomorphic to $P_5$.
The family of $P_5$-free graphs have attracted noticeable interest among researchers, see for instance \cite{bt,lz,mm,rs,sch}.

Similar to cographs, there are eigenvalue-free intervals for a large subclass of $P_5$-graphs.
The following theorem essentially was first proved in \cite{aas} where it was shown that chain graphs (i.e., bipartite graphs in which the neighborhoods of vertices in each color class form a chain with respect to inclusion) have no eigenvalues in the interval $(0,1/2)$ (and hence no eigenvalue in the interval $(-1/2,0)$, as the eigenvalues of bipartite graphs are symmetric with respect to zero). However, as shown in the proof of the following theorem, connected bipartite $P_5$-free graphs are the same as chain graphs. Here we give a simple proof for this result.

\begin{thm}\label{P5} {\rm(\cite{aas})} Bipartite $P_5$-free graphs have no eigenvalue in the intervals $(-1/2,0)$ and $(0,1/2)$.
\end{thm}
\begin{proof}{Eigenvalues of bipartite graphs are symmetric with respect to  0. So is suffices to prove that bipartite $P_5$-free graphs have no eigenvalue in $(0,1/2)$.
The proof goes by induction on the number of vertices. The assertion holds for bipartite graphs with at most 4 vertices (see \cite[p.~17]{bh}). It suffices to consider connected graphs. So let $G$ be a connected bipartite $P_5$-free graph with at least 5 vertices. We claim that the neighborhoods of vertices in the same part are comparable with respect to inclusion; otherwise $G$ contains an induced matching of size 2.
As $G$ is connected, there exits an induced path of length at least $2$ between these two edges which implies the existence of an induced path of length at least $5$ in $G$, a contradiction.

First assume that $G$  has a pair of duplicates $u,v$ and $H=G-v$.
Let $\la_1\ge\cdots\ge\la_\ell$ and $\mu_1\ge\cdots\ge\mu_{\ell-1}$ be the eigenvalues of $G$ and $H$, respectively.
Also suppose that $\mu_t>\mu_{t+1}=\cdots=\mu_{t+j}=0>\mu_{t+j+1}$ (with possibly $j=0$).
By the induction hypothesis, $\mu_t>1/2$ (the equality is impossible as $1/2$ is not an algebraic integer).
By interlacing, we have $\la_{t+1}\ge0=\la_{t+2}=\cdots=\la_{t+j}=0\ge\la_{t+j+1}\ge\mu_{t+j+1}$.
Note that $\mul(0,G)=\mul(0,H)+1=j+1$. This is  possible only if both $\la_{t+1}$ and $\la_{t+j+1}$ are zero.
On the other hand, again by interlacing, $\la_t\ge\mu_t>1/2$. Hence $G$ has no eigenvalue in $(0,1/2)$.

Now, suppose that $G$ has no pair of duplicates.
Let $\{x_1,\ldots,x_n\}$ and $\{y_1,\ldots,y_m\}$ be the parts of $G$.
We showed that the neighborhoods of vertices in each part of $G$ form a chain. This is only possible if the parts of $G$ have the same size $m=n$,
 and the neighborhoods of the vertices are of the form
 $$\{x_1\},\{x_1,x_2\},\ldots,\{x_1,\ldots,x_n\},~~\{y_1\},\{y_1,y_2\},\ldots,\{y_1,\ldots,y_n\}.$$
 It follows that
 $$A(G)=\begin{pmatrix}O&C\\C^\top&O\end{pmatrix},$$
 with $C+C^\top=J_n+I_n$ where  $J_n$ is the all 1's $n\times n$ matrix.
 We have that
$$(2C-I)(2C-I)^\top=4CC^\top-2C-2C^\top+I=4CC^\top-I-2J.$$
This means that $4CC^\top-I=(2C-I)(2C-I)^\top+2J$ is positive semidefinite and so the eigenvalues of $CC^\top$ are not smaller than $1/4$. It turns out that $G$ has no eigenvalue in the interval $(-1/2,1/2)$. This completes the proof.
}\end{proof}

The assumption of `bipartiteness' cannot be removed from Theorem~\ref{P5}. The graph depicted in Figure~\ref{fig}, is non-bipartite and $P_5$-free with an eigenvalue about $-0.462$ (see Table 1 in Appendix of \cite{cds}).
\begin{figure}
  \centering\includegraphics[width=3.5cm]{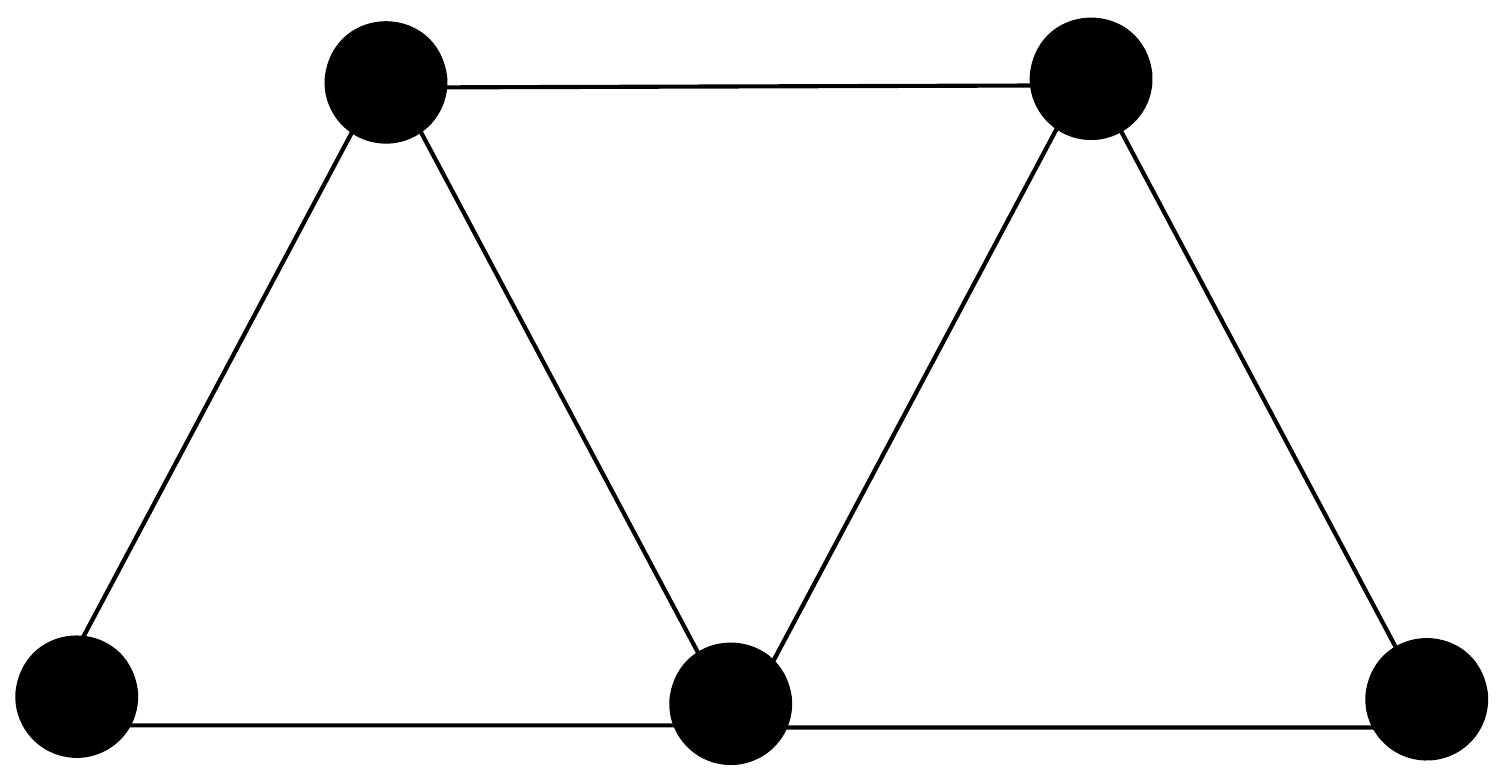}\\
  \caption{A non-bipartite $P_5$-free graph with an eigenvalue in the interval $(-1/2,0)$}\label{fig}
\end{figure}

Also there is no characterization similar to Theorem~\ref{(-1,0)}  for bipartite $P_5$-free graphs. For this, consider $C_6$, the cycle of order 6. It is not $P_5$-free, however it is easy to verify (see Tables 1 and 4 in Appendix of \cite{cds}) that any induced subgraph of $C_6$ has no eigenvalue in the interval $(0,1/2)$.

\section*{Acknowledgments}
The research of the author was in part supported by a grant from IPM (No. 94050114).
He is also indebted to Jack Koolen for fruitful discussion and critical remarks.

{}

\end{document}